\documentclass{amsart}
\usepackage[margin=2cm]{geometry}

\usepackage{graphics, amsmath,amssymb, enumitem, comment, url, mathtools}
\usepackage{graphicx}
\usepackage{epsfig}

\newif\ifcolorcomments
\newcommand{\allowcomments}[4]{
\newcommand{#1}[1]{\ifdraft{\ifcolorcomments{\textcolor{#4}{##1 --#3}}\else{\textsl{ ##1 \ --#3}}\fi}\else{}\fi}
}

\colorcommentstrue
\usepackage{amssymb}
\usepackage{amsmath,amsthm}

\usepackage{colortbl}

\newtheorem{theorem}{Theorem}[section]
\newtheorem{lemma}[theorem]{Lemma}
\newtheorem{proposition}[theorem]{Proposition}
\newtheorem{corollary}[theorem]{Corollary}
\theoremstyle{definition}

\newtheorem{remark}[theorem]{Remark}

\newtheorem{conjecture}[theorem]{Conjecture}

\newcommand{\C}{\mathbb C}
\newcommand{\CC}{\mathcal C}
\newcommand{\CCc}{\mathcal{C}^{\circ}}

\newcommand{\DD}{\mathcal D}

\newcommand{\mfF}{\mathfrak{F}}
\newcommand{\sfF}{\mathsf{F}}

\newcommand{\KK}{\mathcal K}

\newcommand{\N}{\mathbb N}

\newcommand{\Q}{\mathbb Q}

\newcommand{\R}{\mathbb R}

\newcommand{\Z}{\mathbb Z}

\newcommand{\mfFc}{\mathfrak{F}^{\circ}}

\newcommand{\bfa}{\mathbf{a}}
\newcommand{\bfb}{\mathbf{b}}
\newcommand{\bfc}{\mathbf{c}}
\newcommand{\bfh}{\mathbf{h}}
\newcommand{\bfu}{\mathbf{u}}

\newcommand{\bfx}{\mathbf{x}}
\newcommand{\bfw}{\mathbf{w}}

\DeclareMathOperator{\real}{Re}
\DeclareMathOperator{\imag}{Im}

\renewcommand{\text}{\textup}

\newcommand{\NPC}[1]{\ignorespaces}

\renewcommand{\AA}{\mathcal A}

\DeclareMathOperator*{\Exact}{Exact}

\newif\ifdraft\drafttrue

\def\N{\mathbb N}
\def\Z{\mathbb Z}
\def\Q{\mathbb Q}

\def\R{\mathbb R}

\def\e{\boldsymbol\eta}

\IfFileExists{marvosym.sty}{
\RequirePackage{marvosym} 
}{

}

\IfFileExists{wasysym.sty}{
\RequirePackage{wasysym}
}{

}

\mathtoolsset{showonlyrefs}
\newcommand*{\myDots}{\ifmmode\mathellipsis\else.\kern-0.07em.\kern-0.07em.\fi}

\usepackage[dvipsnames]{xcolor}

\allowcomments{\commumtaz}{MH}{Mumtaz}{green}
\allowcomments{\comnikita}{NS}{Nikita}{blue}
\allowcomments{\comgero}{GGR}{Gero}{red}

\newcommand {\ignore}[1] {}

\begin{document}

\title{Complex numbers with a prescribed order of approximation and Zaremba's conjecture}

\author[G. Gonz\'alez Robert]{Gerardo Gonz\'alez Robert}
\author[M. Hussain]{Mumtaz Hussain}
\author[N. Shulga]{Nikita Shulga}
\address{Department of Mathematical and Physical Sciences,  La Trobe University, Bendigo 3552, Australia. }
\email{G.Robert@latrobe.edu.au}
\email{m.hussain@latrobe.edu.au}
\email{n.shulga@latrobe.edu.au}


\date{}

\maketitle

\begin{abstract}
Given $b=-A\pm i$ with $A$ being a positive integer, we can represent any complex number as a power series in $b$ with coefficients in $\AA=\{0,1,\ldots, A^2\}$. We prove that, for any real $\tau\geq 2$ and any non-empty proper subset $J(b)$ of $\AA$, there are uncountably many complex numbers (including transcendental numbers) that can be expressed as a power series in $b$ with coefficients in $J(b)$ and with the irrationality exponent (in terms of Gaussian integers) equal to $\tau$. 
One of the key ingredients in our construction is the `Folding Lemma' applied to Hurwitz continued fractions. This motivates a Hurwitz continued fraction analogue of the well-known Zaremba's conjecture. We prove several results in support of this conjecture.

\end{abstract}

\section{Introduction}

Real numbers are commonly represented in integer base expansions. It is thus natural to consider a similar representation of complex numbers in terms of a Gaussian integer base. However, this proves to be a more intricate problem compared to its real number counterpart.  Specifically, suppose we have a Gaussian integer $b$  such that $|b|>1$ and let $\AA$ be a system of representatives $\pmod b$. When $\real(b)$ and $\imag(b)$ are coprime, we may consider $\AA=\{0,1,\ldots, |b|^2-1\}$ (see, for example, \cite[Theorem 3.10.4]{AlloucheShallit2003})). Under this assumption, K\'atai and Szab\'o showed that for each Gaussian integer $z$ there exists a unique sequence $(d_j)_{j\geq 0}$ in $\AA$ satisfying 
\[
\lim_{n\to \infty} d_j=0
\quad\text{ and }\quad
z = \sum_{j=0}^{\infty} d_j b^j
\]
if and only if $b=-A\pm i$ for some $A\in \N$ (see \cite[Theorem 1]{KataiSzabo1975} and \cite{PethoThuswaldner2018} for a generalisation). Furthermore, in the same paper, they proved under the same assumptions that for each complex number $\zeta$ there is a sequence $(d_j)_{j\in \Z}$ in $\AA$ satisfying
\[
\lim_{n\to-\infty} d_j=0
\quad\text{ and }\quad
\sum_{j\in \Z}^{l} d_jb^{-j}.
\]
These expansions lead to situations not encountered in the real case. For example, the boundary of the set 
\[
\left\{ \sum_{j\in \N} \frac{d_j}{b^j}: \forall j\in \N \; \quad d_j\in \{0,1,\ldots, A^2\} \right\}
\]
exhibits a fractal structure with the Hausdorff dimension in the interval $(1,2)$ (see \cite{AkiyamaThuswaldner2004, Gilbert1986}).

In this paper, we restrict assume that $b=-A\pm i$ for some $A\in \N$.
By $J(b)$ we mean a proper subset of $\{0,1,\ldots, A^2\}$ and we define the set
\[
\KK_{J(b)}:=
\left\{ \zeta=\sum_{j\in \N} d_j b^{-j}:  \  \ d_j\in J(b)  \text{ for all } j\in\N  \right\}.
\]
Although there are uncountably many complex numbers with more than one expansion in base $b$, we can easily find non-empty sets  $J(b)\subset \{0,1,\ldots, A^2\}$ for which $\KK_{J(b)}$ is a Cantor set. Our first goal is to study approximation properties of elements in $\KK_{J(b)}$. To this end, for each positive function $\Psi:\R_{\geq 1} \to \R_{>0}$, define the set of $\Psi$-approximable numbers to be
\[
\KK(\Psi)
:=
\left\{ \zeta\in \C: \left| \zeta - \frac{p}{q}\right| < \Psi(|q|) \text{ for infinitely many } \frac{p}{q}\in \Q(i)\right\}.
\]
The set of complex numbers exactly approximable to order $\Psi$ is defined to be
\[
\Exact(\Psi)
:=
\KK(\Psi)\setminus \bigcup_{k\in\N} \KK\left( \left(1 - \frac{1}{k+1}\right) \Psi\right).
\]
The irrationality exponent $\mu(\xi)$ of a complex number $\xi\in \C\setminus \Q(i)$ is given by
\[
\mu(\xi):=
\sup\left\{ \mu\in \R: \left| \xi - \frac{p}{q}\right| < \frac{1}{|q|^{\mu}} \text{ for infinitely many } \frac{p}{q}\in \Q(i)\right\}.
\]
The definitions of $\KK(\Psi)$ and $\mu(\xi)$ are obvious analogues of those for real numbers, where we consider approximations by numbers in $\Q$ rather than in $\Q(i)$. In this context, a famous problem by Mahler \cite[Section 2]{Mahler1984} is to determine how well can irrational numbers in the middle-third Cantor set $K$ be approximated by rationals. Expressed differently, he asked whether $K$ contains numbers with a prescribed irrationality exponent $\tau$. Levesley, Salp, and Velani \cite{LevesleySalpVelani2007} solved Mahler's question by showing that the set of real numbers in $K$ whose irrationality exponent is strictly larger than two, has a positive Hausdorff dimension. They also gave examples of numbers (in fact transcendental numbers) in $K$ with the irrationality exponent larger than $\frac{3+\sqrt 5}{2}$. Shortly afterward, Bugeaud \cite{BugeaudDACS2008} constructed numbers in $K$ with any prescribed irrationality exponent $\tau\geq 2$. In the same paper, Bugeaud proposed a stronger version of Mahler's problem: 

{\em Given a positive function $\Psi$, is $K\cap \Exact(\Psi)\neq \varnothing$ true?} 

In the complex plane, Bugeaud's strengthening of Mahler's problem can be extended as follows: 

{\em for $b$ and $J(b)$ as above and a positive function $\Psi$, is the set $\Exact(\Psi)\cap \KK_{J(b)}$ non-empty?} 

Our first result extends Bugeaud's answer to complex numbers for a slightly weaker problem.

\begin{theorem}\label{Te:Gero:01}
Let $b \in\Z[i]$ be of the form $b=-A\pm i$ with $A\in \N$, let $J(b)$ be any proper subset of $\{0,1,\ldots, A^2\}$, and let $\Psi:\R_{>1}\to \R_{\geq 0}$ be a function such that
 $x^2\Psi(x)\to 0$ as $x\to \infty$, and $x\mapsto x^2\Psi(x)$ is decreasing. Then, for any number $c\in (0, |b|^{-1})$, the set $\left( \KK(\Psi)\setminus \KK(c\Psi)\right) \cap K_{J(b)}$ is non-empty. Moreover, when $A\geq 2$, the intersection is uncountable. 
\end{theorem}

The set $\Exact(\Psi)$ in the complex case was considered by He and Xiong \cite[Theorem 1.2]{HeXiong2022}. They calculated its Hausdorff and packing dimensions for any $\Psi$ satisfying the hypothesis of Theorem \ref{Te:Gero:01}. As stated above, in the real case, the exact approximation sets were comprehensively studied by Bugeaud \cite{MR2006007, BugeaudDACS2008, MR2439605} under similar asymptotic conditions on $\Psi$.

The theory of Hurwitz continued fractions tells us that $\mu(\xi)\geq 2$ for all $\xi\in\C\setminus\Q(i)$ (see statement \ref{Propo:HCFApproximation:03} in Proposition \ref{Propo:HCFApproximation}). Now by considering $\tau\geq 2$, $\Psi(x)= x^{-\tau} (\log(1+x))^{-1}$, and $J(b)=\{0,1\}$, Theorem \ref{Te:Gero:01} implies that $\KK_{J(b)}$ contains irrational numbers $\xi$ such that $\mu(\xi)=\tau$. The proof of Theorem \ref{Te:Gero:01} gives us some of those numbers $\xi$ explicitly, which we can formulate in the following way.

\begin{theorem}\label{Te:Gero:02}
Let $b$ be a Gaussian integer of the form $b=-A \pm i$ with $A\in \N_{\geq 2}$. For each $\lambda>0$ define
\[
n_0(\lambda,\tau)
:=
1 + \max\left\{\ 0, \left\lfloor \frac{\log(3/\lambda)}{\log \tau} \right\rfloor, \frac{\log 3}{\log |b|} \right\}.
\]
Then, the complex number
\[
\xi(\tau,\lambda)
=
\sum_{n\geq n_0(\lambda,\tau)} \,\frac{1}{b^{\lfloor \lambda\tau^n \rfloor}}\in K_{\{0,1\}}
\] 
has irrationality exponent $\mu(\xi(\tau, \lambda))=\tau$. Furthermore, when $\tau=2$, the numbers $\xi(\tau,\lambda)$ are badly approximable. 
\end{theorem}
\begin{corollary}\label{Coro:PrescribedExponent}
Let $\tau\geq 2$ be arbitrary. If $b=-A\pm i$ for some $A\in\N$ and $J(b)=\{0,1\}$, then $\KK_{J(b)}$ contains uncountably many complex numbers with irrationality exponent $\tau$.
\end{corollary}
\begin{remark}
When $A=1$, we have $\KK_{J(b)}=\C$ by \cite[Theorem 2]{KataiSzabo1975}. In this case, Corollary \ref{Coro:PrescribedExponent} follows trivially. 
\end{remark}

The proofs of Theorem \ref{Te:Gero:01} and Theorem \ref{Te:Gero:02} use the theory of Hurwitz continued fractions ({\bf HCFs}). The literature on HCFs is not yet fully developed akin to regular real continued fractions ({\bf RCFs}), however, there has been some progress in developing the metrical properties of HCFs to study sets of complex numbers defined by Diophantine approximation conditions. When proving Theorems \ref{Te:Gero:01} and \ref{Te:Gero:02}, we need Gaussian rationals of the form $r/b^v$ with $v\in \N$, whose HCFs satisfy certain properties. This is not an easy task when we impose an upper bound independent of $b$ on the partial quotients. However, the HCF analogue of the well-known Zaremba's conjecture (1971) asserts that we can always do it. 

\begin{conjecture}[Zaremba's Conjecture for Hurwitz continued fractions] There exists an absolute constant $\e\in\R$ with the following property:
 for any $b\in\Z[i]\backslash\{0, \pm 1, \pm i\}$, there exists $a\in\Z[i]\backslash\{0\}$ with $\gcd(a,b)=1$ and $a/b\in\mfF$, with $\mfF$ from \eqref{unit square}, such that if $a/b=[a_1,\ldots,a_n]_\C$ is a HCF, then $\max_i |a_i|~\leq~\e$.
\end{conjecture}

We prove Zaremba's conjecture in several special cases (in Section \ref{SEC:Zaremba}): for denominators of the form
\begin{itemize}
    \item  $b=(-3\pm i)^n,(-2\pm i)^n,\, n\geq1$ with $\e=3\sqrt2$, see Theorems \ref{zarembafor-3+i} and \ref{zarembafor-2+i}.
    \item $b=2^n,3^n,\, n\geq1$ with $\e=8$, see Theorems \ref{zarembafor2} and \ref{zarembafor3}.
    \item $b=5^n,\,n\geq1$ with $\e=7$, see Theorem \ref{zarembafor5}.
\end{itemize}


In the context of real numbers and regular continued fractions, Niederreiter proved Zaremba's Conjecture for powers of $2$ and $3$ with $\e=3$ in \cite{Niederreiter1986}. In recent years, Zaremba's conjecture has been extensively studied. We refer to \cite{MR3194813,  MR3361774, MR4145818, MR4202008, racicalzaremba} 
for some other results on Zaremba's conjecture.

Note that the bounds $\e$ which we got for the HCF analogue of Zaremba's conjecture are weaker than their real counterparts. We believe this should be the case by the nature of HCF. For RCFs, every sequence of natural numbers is a sequence of partial quotients of some number $x\in\R$. For a sequence of Gaussian integers and HCFs, analogous statement is not true. First of all, HCF partial quotients cannot take values from $\{0,\pm1,\pm i\}$. Furthermore, there are some non-trivial dependence relations between successive partial quotients,  which pose some restrictions disproportionally affecting partial quotients with small absolute values. As Zaremba's conjecture deals with very small partial quotients (in fact, in the real case it is conjectured that even $\e=2$ should work for large enough prime numbers $q$ as denominators), HCF expansions are naturally forced to take larger partial quotients more often.
For more details, see Section \ref{SEC:HCF} and Remark \ref{remark2n} below.

The paper is organised in the following manner. In Section \ref{SEC:Folding Lemma}, we recall of one of our key tools: Folding Lemma. Section \ref{SEC:HCF} contains notation, terminology, and basic facts on HCFs. Theorems \ref{Te:Gero:01} and \ref{Te:Gero:02} are proven in Section \ref{SEC:IrrationalityExp}. In Section \ref{SEC:Zaremba} we give a formal statement and proofs of our results around Zaremba's Conjecture for HCFs. Lastly, in Section \ref{validitycheck} we check that the continued fractions obtained by a repeated application of Folding Lemma are indeed HCFs.

\noindent{\bf Acknowledgments:} This research is supported by the Australian Research Council discovery project grant number 200100994. We are thankful to Johannes Schleischitz for posing this question during the Trends in Metric Number Theory Conference. 
\section{The continued fraction apparatus}\label{SEC:Folding Lemma}
A continued fraction is an object of the form
\begin{equation}\label{defgeneral}
a_0 + \cfrac{1}{a_1+\cfrac{1}{a_2+\cfrac{1}{a_3+\ddots}}},
\end{equation}
which can conveniently be denoted as $[a_0; a_1,\ldots,a_n,\ldots]$. When $a_0=0$, we will sometimes write $[a_1,\ldots,a_n,\ldots]:=[0; a_1,\ldots,a_n,\ldots]$. The coefficients $a_i$'s in different contexts can be very different things, for example, positive integers, non-zero integers, polynomials, complex numbers, and many others. Denote a finite truncation of \eqref{defgeneral} as
$$
\frac{p_n}{q_n}=[a_0;a_1,\ldots,a_n].
$$
It is well-known (see, for example, \cite[Theorem 1]{Khinchin_book}) that the sequences $(p_n)_{n\geq 0}$ and $(q_n)_{n\geq 0}$ satisfy the following recursion formulae
\begin{align}\label{Eq:Recurrence}
p_{-1}&=1, & p_0 &= a_0, & p_n&=a_n p_{n-1}+p_{n-2};\\
q_{-1}&=0,  & q_0 &= 1,  & q_n&=a_n q_{n-1}+q_{n-2}.
\end{align}
We refer to the sequences $(p_n)_{n\geq 0}$ and $(q_n)_{n\geq 0}$ as the sequences of numerators and denominators of $[a_0;a_1,a_2,\ldots]$ respectively. Next, we gather some properties of continued fractions that can be found in standard books, see for example \cite[Chapter 1]{Khinchin_book}.

\begin{proposition}
Take $n\in\N$ and a continued fraction $[a_0;a_1,\ldots, a_n]$, then
\begin{align*}
    q_np_{n-1}-q_{n-1}p_n &= (-1)^n, \\
    \frac{p_n \alpha+p_{n-1}}{q_n\alpha+q_{n-1}} &=[a_0;a_1,\ldots,a_n,\alpha  ] \text{ for a variable }\alpha\neq0, \\
    \frac{q_{n-1} }{q_n} &= [a_n,\ldots,a_1  ],\\
    x - \frac{q_{n-1} }{q_n}  &=[x;-a_n,\ldots,-a_1  ] \text{ for a variable }x.
\end{align*}
\end{proposition}
Next, we state the remarkable folklore result known as the ``Folding Lemma''.
\begin{lemma}[\cite{MR1149740}, Proposition 2]\label{folding} If $p_n/q_n=[a_0;a_1,\ldots, a_n]$, then for any nonzero $x$,we have
   \begin{equation}\label{folding:eq}
\frac{p_n}{q_n} + \frac{(-1)^n}{xq_n^2}= \left[  a_0; a_1,\ldots,a_n,x,-a_n,\ldots,-a_1        \right].
\end{equation}
\end{lemma} 




Given the importance of this lemma in the proofs to follow,  we include its proof.
\begin{proof}
Certainly,
\begin{align*}
\frac{p_n}{q_n} + \frac{(-1)^n}{xq_n^2} 
&= \frac{p_n}{q_n}+\frac{q_np_{n-1}-q_{n-1}p_n}{x q_n^2} = \frac{xq_np_n+q_np_{n-1}-q_{n-1}p_n}{xq_n^2}\\
&=\frac{p_n(xq_n-q_{n-1})+q_n+p_{n-1}}{q_n(xq_n)}=\frac{p_n(x-q_{n-1}/q_n)+p_{n-1}}{xq_n+q_{n-1}+q_{n-1}} \\
&=\frac{p_n(x-q_{n-1}/q_n)+p_{n-1}}{q_n(x-q_{n-1}/q_n)+q_{n-1}}=\left[a_0;a_1,\ldots,a_n,x-\frac{q_{n-1}}{q_n} \right]\\ &=\left[  a_0; a_1,\ldots,a_n,x,-a_n,\ldots,-a_1        \right].
    \end{align*}
\end{proof}

\section{Hurwitz Continued Fractions}\label{SEC:HCF}
In this section we introduce Hurwitz continued fractions and some of their basic properties. Let $\lfloor \,\cdot\,\rfloor:\R\to\Z$ be the floor function. The \textbf{nearest Gaussian integer} to a given complex number $z\in \C$ is 
\[
[z]
:=
\left\lfloor \real(z) + \frac{1}{2}\right\rfloor
+
i \left\lfloor \imag(z) + \frac{1}{2}\right\rfloor.
\]
Let $\mfF$ be the the unit square centered given by
\begin{equation}\label{unit square}
\mfF
=
\{z\in \C: [z]=0\}
=
\left\{ z\in \C: -\frac{1}{2}\leq \real(z)< \frac{1}{2}, \; -\frac{1}{2}\leq \imag(z)< \frac{1}{2}\right\}.
\end{equation}
The \textbf{complex Gauss map}, $T:\mfF\to \mfF$, is the function given by
\[
\forall z\in \C
\quad
T(z) = 
\begin{cases}
z^{-1}- [z^{-1}], \text{ if } z\neq 0,\\
0, \text{ if } z=0.
\end{cases}
\]
Let $T^0$ be the identity map on $\mfF$ and put $T^{n}:=T\circ T^{n-1}$ for all $n\in\N$. Define $a_1:\mfF\setminus\{0\}\to\Z[i]$ by $a_1(z) \colon= [z^{-1}]$ and, while $T^{n-1}(z)\neq 0$, write $a_n(z):= a_1(T^{n-1}(z))$. The \textbf{Hurwitz Continued Fraction} of $z$ is $[a_1(z),a_2(z),\ldots]_\C$. The terms of the sequence $(a_n)_{n\geq 1}$ are the \textbf{partial quotients} of $z$. We extend the definition of HCFs to $\C$ by putting $a_0(z)=[z]$ and $a_n(z):=a_n(z-a_0)$, $n\in\N$. When introducing this continued fraction, Hurwitz showed that a complex number $\zeta$ belongs to $\C\setminus \Q(i)$ if and only if its sequence of partial quotients is infinite and, in this case,
\[
\zeta=\lim_{n\to\infty} \frac{p_n}{q_n}
\]
\cite{Hurwitz1887}. The partial quotients of an HCF may only take values in the set $\DD:=\Z[i]\setminus\{0,1,i,-1,-i\}$ (see for example \cite[Section 6]{DaniNogueira2014}). We will also use $\DD^{+}$ to denote the set of finite and non-empty sequences on $\DD$. We refer to ``$a_i\in \DD$'' as the trivial restriction. For any $n\in\N$ and each $\bfb=(b_1,\ldots,b_n)\in\DD^n$, the \textbf{cylinder} of level $n$ based on $\bfb$ is defined as
\[
\CC_n(\bfb):=
\{z \in\mfF: a_1(z)=b_1, \ldots, a_n(z)=b_n\}.
\] 
The \textbf{prototype set} based on $\bfb$ is $\mfF_n(\bfb):= T^n\left( \CC_n(\bfb) \right)$. By $\CC^{\circ}_n(\bfb)$ we mean the interior of $\CC_n(\bfb)$ and we define $\mfFc_n(\bfb)$ and $\mfFc$ analogously. We say that $\bfb$ is  \textbf{valid} (resp. \textbf{full}) if $\CC_n(\bfb)\neq\varnothing$ (resp. $\mfF_n^{\circ}( \bfb )=\mfF^{\circ}$). Let $\Omega(n)$ (resp. $\sfF(n)$) be the set of all valid (resp. full) words of length $n$, thus $\sfF(n)\subseteq \Omega(n)$. It is easy to check that there are exactly thirteen different non-empty sets of the form $\mfFc_n(\bfa)$ (see, for example, \cite[Section 2]{EiItoNakada2019}). We say that $\bfb=(b_1,\ldots, b_n)\in\Omega(n)$ is \textbf{reversible} if $\overleftarrow{\bfb}:=(b_n,\ldots, b_1)\in \Omega(n)$. 

In general, it might be hard to determine whether a finite sequence is valid or not. However, this occurs when the digits are large enough.

\begin{proposition}[{\cite[Proposition 4.3]{BugeaudGeroHussain2023} }]\label{Prop:FullCyls}  
Let $\bfa=(a_j)_{j\geq 1}$ be a sequence in $\DD$ such that $\min_j|a_j|\geq \sqrt{8}$.
    \begin{enumerate}[label=\normalfont(\arabic*)]
        \item \label{Prop:FullCyls:01} If $\mathbf{a}$ is finite, say $\bfa=(a_1,\ldots, a_n)$ for $n\in\N$,  then $\mathbf{a}\in \sfF(n)$; in particular, $\bfa\in \Omega(n)$. 
        \item \label{Prop:FullCyls:02} If $\bfa=(a_j)_{j\geq 1}$ is infinite, then $\zeta=[a_1,a_2,\ldots]_{\C}$ for some $\zeta\in \C\setminus\Q(i)$.
    \end{enumerate}
\end{proposition}
Consider $n\in\N$ and let $(a_1,\ldots, a_n)\in \DD^n$ be such that $\CCc_n(a_1,\ldots, a_n)\neq \varnothing$. When the assumptions of statement \ref{Prop:FullCyls:01} in Proposition \ref{Prop:FullCyls} are not met, there might exist Gaussian integers $a_{n+1}\in\DD$ for which $\CCc_{n+1}(a_1,\ldots, a_{n+1}) = \varnothing$ or even $\CC_{n+1}(a_1,\ldots, a_{n+1}) = \varnothing$. For example, $\CC_2(-1+2i,1+i)=\varnothing$, and $\CC_2(-2,1+4i)\neq \varnothing$ but $\CCc_2(-2,1+4i)= \varnothing$. Table \ref{Table:HensleyTable}, inspired by one in \cite[p. 75]{Hensley_book}, tells us how to choose $a_{n+1}$ so that $\CCc_{n+1}(a_1,\ldots, a_{n+1}) \neq \varnothing$. 
\begin{table}[h!]
\centering
\begin{tabular}{| c | c |} 
 \hline
 Conditions on $a_n$ & Maximal subset of $\DD$ containing $a_{n+1}$ \\ [0.5ex] 
 \hline
 $\max\{|\real(a_n)|,|\imag(a_n)|\}\geq 3$ & $\DD$ \\
 $1 + i$ & $\{a\in \DD:\imag(a_{n+1})<0<\real(a_{n+1})\}$ \\ 
  $2$ & $\{a\in \DD:\real(a_{n+1})\geq 0\}$ \\
 $2 + i$ & either $\DD\setminus\{ -1+i\}$  or $\{a\in \DD:\real(a)\geq 0\}$ \\
 $-1 + 2i$ & either $\DD\setminus\{ 1+i\}$  or $\{a\in \DD:\real(a)\leq 0\}$  \\ [1ex] 
 \hline 
\end{tabular} 
\caption{Non-trivial restrictions.}
\label{Table:HensleyTable}
\end{table}
The table remains true if we multiply both columns by the same integral power of $i$ or if we apply complex conjugation. 
Some care is needed in the last two lines, though. When $|a_n|=\sqrt{5}$, the open prototype set $\mfFc_{j}(a_1,\ldots, a_j)$ may have two different forms. For instance, for $2+i$ we have
 \[
\mfFc_2(2+2i, 2+ i)=\mfF_1^{\circ}(2+i)
\quad\text{ and }\quad
 \mfFc_2(-2 + i, 2+ i)=\mfF_1^{\circ}(2).
 \]
As a consequence, note that while $(2 + 2i, 2+i, -3+4i)$ is valid, $(-2 + 2i, 2+i, -3+4i)$ is not. 


In the next proposition, we collect some important properties of HCFs.
\begin{proposition}\label{Propo:HCFApproximation}
Take $z\in\mfF\setminus\{0\}$ and let $(p_n)_{n\geq 1}$, $(q_n)_{n\geq 1}$ be as in \eqref{Eq:Recurrence}. 
\begin{enumerate}[label=\normalfont(\arabic*)]
\item \label{Propo:HCFApproximation:01} \cite[Theorem 6.1]{DaniNogueira2014}. The sequence $(|q_n|)_{n\geq 0}$ is strictly increasing. 
\item \label{Propo:HCFApproximation:02} \cite[Proposition 3.3]{DaniNogueira2014}. If $z=[a_1,a_2,\ldots]_\C \in\mfF\setminus\Q(i)$, then
\[
\text{for all \ } n\in\N,
\quad
z= \frac{(a_{n+1}+[a_{n+2},a_{n+3},\ldots]_\C)p_{n}+p_{n-1}}{(a_{n+1}+[a_{n+2},a_{n+3},\ldots]_\C)q_{n}+q_{n-1}}.
\]
\item \label{Propo:HCFApproximation:03} \cite[Theorem 1]{Lakein1973}. For every $n\in\N$, we have
    \[
    \left| \zeta -  \frac{p_n}{q_n}\right| < \frac{1}{|q_n|^2}.
    \]
\item \label{Propo:HCFApproximation:04} \cite[Lemma 2.2]{HeXiong2022}. Take $p,  q\in \Z[i]$ with $q\neq 0$. If
    \[
    \left| \zeta -  \frac{p}{q}\right| < \frac{1}{4|q|^2},
    \]
    then $\frac{p}{q}=\frac{p_n}{q_n}$ for some $n\in\N$.
\end{enumerate}
\end{proposition}

When restricted to real numbers, HCFs partial quotients may take values in $\Z\setminus\{-1,0,1\}$. Hence, obviously, it may not coincide with the RCFs expansion. For reference, we state the next remark.

\begin{remark}\label{remarkdiff}
Let $\bfa=(a_1,\ldots,a_n)\in\Z[i]^n$. The following statements are true.
    \begin{enumerate}
        \item\label{remarkdiff1} 
  If $[a_1,\ldots,a_n]_\R$ is a RCF of a real number $x\in[0,1]$, then $[a_n,\ldots,a_1]_\R$ is also the RCF of some real number $y\in[0,1]$.

   \item\label{remarkdiff2}  If $\bfa\in\Omega(n)$, then in general it is not true that $\bfa$ is reversible. 

  \item\label{remarkdiff3}   If $\bfa\in\Omega(n)$ and for all $i$ one has $a_i\in\Z$, then $\bfa$ is reversible if and only if for every $2\leq k \leq n$ such that $|a_k| =2$, we have $a_{k-1}a_k>0$. 

  \item\label{remarkdiff4}  (Corollary of the previous statement)  If $\bfa\in\Omega(n)$ and for all $1\leq i \leq n$ we have $|a_i|\geq3,a_i\in\Z$, then $\bfa$ is reversible.
  \end{enumerate}
\end{remark}

The statement \eqref{remarkdiff3} of Remark \ref{remarkdiff} can also be seen as a corollary of Proposition 1.1 from \cite{MR3692894}, where Simmons investigated properties of  HCFs the of real numbers. We refer to his paper for more details.

\section{Proofs of Theorem \ref{Te:Gero:01} and Theorem \ref{Te:Gero:02}} \label{SEC:IrrationalityExp}
We start off by providing some estimates on the approximation by the $n$th convergent. Next, we construct a family of complex numbers depending on certain parameters using Folding Lemma. Finally, we obtain Theorem \ref{Te:Gero:01} and Theorem \ref{Te:Gero:02} by appropriately choosing the values of the parameters and checking that the conclusions are true.

Lakein \cite[Theorem 2]{Lakein1973} proved that Hurwitz continued fractions provide good Gaussian rational approximations to a given number $\zeta\in \C$ in the sense that, for any $n\in \N$,
\[
|q_n\zeta - p_n|
=
\min\left\{ |q'\zeta - p'|: p',q'\in\Z[i], \, |q'|\leq |q|\right\}.
\]
However, some good approximations might be overlooked by Hurwitz's algorithm (see \cite{LeVeque1952} for a short discussion and \cite[Section 2]{MR3692894}). In Theorem \ref{Te:Gero:01} and in Theorem \ref{Te:Gero:02} under the assumption $\tau>2$, we are considering the approximation function $\Psi(x)=x^{-\tau}$ satisfying $x^2\Psi(x)\to 0$ as $x\to \infty$. In particular, if $q\in \Z[i]$ has a large absolute value, then $|q|^2\Psi(|q|)< 4^{-1}$ and, for a given $\zeta\in \C$, 
 \[
 \left| \zeta - \frac{p}{q}\right|<\Psi(q)
\quad \text{ implies }\quad
 \left| \zeta - \frac{p}{q}\right|<\frac{1}{4|q|^2}
 \]
so, by Proposition \ref{Propo:HCFApproximation}, $p/q=p_n/q_n$ for some $n\in\N$. The case $\tau=2$ in Theorem \ref{Te:Gero:02} can also be studied using HCFs. Indeed, let $\zeta\in\C$ be given, then $\mu(\zeta)\geq 2$ in the view of Proposition \ref{Prop:Gero:Proof}. Whereas, it is known \cite{GeroPP-2020, Hines2019} that the Hurwitz partial quotients of $\zeta$ are bounded if and only if for some constant $c>0$ we have
\begin{equation}\label{Eq:BadDef}
\forall (p,q)\in \Z[i]\times (\Z[i]\setminus \{0\}), 
\quad
\left| \zeta - \frac{p}{q}\right| \geq \frac{c}{|q|^2}.    
\end{equation}

Hence, the irrationality exponent of such $\zeta$'s is two, that is, $\mu(\zeta)=2$. A complex number $\zeta\in\C\setminus\Q(i)$ satisfying \eqref{Eq:BadDef} is called \textbf{badly approximable}.

\begin{lemma}\label{Le:Proof:Gero:01}
For every $\zeta=[c_0;c_1,c_2,\ldots]_\C\in \C\setminus \Q(i)$ and for all $n\in \N$,  we have 
\[
\left(\frac{2}{2 + \sqrt{2}}\right)\,
\frac{1}{ |q_n|^2 \left( |c_{n+1}| + \frac{\sqrt{2}}{2}\right)}
\leq
\left| \zeta - \frac{p_n}{q_n}\right|
\leq 
\left(\frac{4}{2-\sqrt{2}}\right)\,
\frac{1}{|c_{n+1}||q_n|^2}.
\]
\end{lemma}
\begin{proof}
Take any $n\in\N$. Direct computations give $|1 + q_{n-1}/(c_{n+1}q_n)|\geq 2 - \sqrt{2}/2$ (see \cite[Proposition 7.2]{BugeaudGeroHussain2023}). Then, by part \ref{Propo:HCFApproximation:03} of Proposition \ref{Propo:HCFApproximation},
\begin{align*}
\left| \zeta - \frac{p_n}{q_n}\right|
&\leq 
\left| \zeta - \frac{p_{n+1}}{q_{n+1}}\right| + \left| \frac{p_{n+1}}{q_{n+1}} - \frac{p_n}{q_n}\right|  \\
&<
\frac{2}{|q_nq_{n+1}|^2} 
=
\frac{2}{|c_{n+1} q_n^2| \left|1 + \frac{q_{n-1} }{c_{n+1}q_n } \right|}
<
\left(\frac{4}{2-\sqrt{2}}\right)\, \frac{1}{|c_{n+1}||q_n|^2}.
\end{align*}
For the lower bound, we use part \ref{Propo:HCFApproximation:02} of Proposition \ref{Propo:HCFApproximation} to get
\[
\left| \zeta - \frac{p_n}{q_n}\right|
=
\left| \frac{[c_{n+1},c_{n+2},\ldots]_{\C}}{q_n^2\left( 1 + \frac{q_{n-1}}{q_n}[c_{n+1},c_{n+2},\ldots]_{\C}\right)} \right|.
\]
Hence, the strict growth $(|q_j|)_{j\geq 1}$ and $[c_{n+1}, c_{n+2},\ldots]_{\C}\in \mfF$ gives us
\[
\left| 1 + \frac{q_{n-1}}{q_n}[c_{n+1}, c_{n+2},\ldots]_{\C}\right| \leq 1 + \frac{\sqrt{2}}{2}
\]
and, because of $[c_{n+2}, c_{n+3},\ldots]_{\C} \in\mfF$,
\[
\left| [c_{n+1}, c_{n+2},\ldots]_{\C}\right|
=
\frac{1}{\left| c_{n+1} + [c_{n+2}, c_{n+3},\ldots]_{\C} \right| } \geq \frac{1}{\left| c_{n+1}\right| + \frac{\sqrt{2}}{2}  }.
\]
\end{proof}
\begin{remark}
    For RCFs, the estimates corresponding to Lemma \ref{Le:Proof:Gero:01} are an immediate consequence of $|\alpha - p_n/q_n|< (q_nq_{n+1})^{-1}$ for all $n\in\N$. This estimate, however, may fail for HCFs. For instance, if $q_1, q_2$ are the first two denominators of the HCF of
    \[
    \zeta 
    = 
    \left[-2 + 2i, 2+2i, \frac{-1-i}{4\sqrt{2}}\right],
    \]
    then $|\zeta - p_1/q_1|\geq 1/|q_1q_2|$.
\end{remark}

Consider $b=-A+i\in \Z[i]$ with $A\in\N$, a natural number $v\in \N$, and $r\in \Z[i]$ such that 
\[
\frac{r}{b^v} 
=
[a_1,a_2,\ldots, a_h]_{\C}.
\]
Assume that both $\bfa=(a_1,\ldots, a_h)$ and $\overleftarrow{\bfa}$ belong to $\sfF(h)$. In this framework, to each sequence of natural numbers $\bfu=(u_n)_{n\geq 1}$ we will associate a complex number $\xi_{\bfu}$ by the following procedure.

First, define the sequence $(v_n)_{n\geq 0}$ of natural numbers by
\[
v_0:= v
\qquad\text{ and }\qquad
\forall n\in  \N \quad v_{n}=u_{n} + 2v_{n-1},
\]
so $v_{n}\geq 2^n$ for $n\in \N$. Consider the sequence of complex numbers $(\xi^{j}_{\bfu})_{j\geq 1}$ obtained by a repeated 
 application of Folding Lemma:
\[
\xi^1_{\bfu} := \frac{r}{b^v} + \frac{(-1)^h}{b^{v_1}}
\qquad\text{ and }\qquad
\forall n\in \N \quad \xi^{n+1}_{\bfu} :=\xi^{n}_{\bfu} + \frac{-1}{b^{v_{n+1}}}.
\]
Since $|b|>1$, the following limit exists:
\[
\xi_{\bfu}
:=
\lim_{n\to\infty} \xi^{n}_{\bfu},
\]
For each $\bfb:=(b_1,\ldots,b_{n})\in \DD^n$, $n\in\N$, write $-\bfb:=(-b_1,\ldots,-b_{n})$. Folding Lemma gives us the HCF of each $\xi_{\bfu}^j$ and of $\xi_{\bfu}$ straightaway. Indeed, given $d\in \DD$, let $f_d:\DD^{+}\to\DD^{+}$ be the function mapping $\bfb=(b_1,\ldots,b_n)\in \DD^{+}$ to $f_d(\bfb)= \bfb\,d\,\overleftarrow{-\bfb}$. Then, writing $f_{d_n}\cdots f_{d_1}:=f_{d_n}\circ\cdots\circ f_{d_1}$ for any $n\in\N$ and $d_1,\ldots, d_n\in\DD$, we have
\[
\xi^1_{\bfu} = [f_{b^{u_1}}(\bfa)]_{\C}, \quad
\xi^2_{\bfu} = [f_{b^{u_2}}f_{b^{u_1}}(\bfa)]_{\C}, \quad
\xi^3_{\bfu} = [f_{b^{u_3}}f_{b^{u_2}}f_{b^{u_1}}(\bfa)]_{\C},
\]
and so on. Note that $|f_{d_n}\cdots f_{d_1}(\bfa)|= 2^n(h+1) -1$ for all $n\in\N$ and $(\xi_{\bfu}^j)_{j\geq 1}$ is indeed constructed via repeated applications of Folding Lemma. Finally, if we write
\[
\bfc
:=
(c_n)_{n\geq 1}
=
\lim_{n\to\infty}
f_{b^{u_n}} \cdots f_{b^{u_1}} (\bfa ),
\]
the number $\xi_{\bfu}$ admits the continued fraction expansion $\xi_{\bfu}=[c_1,c_2,c_3,\ldots]$. 
\begin{remark}
In general, $[c_1,c_2,c_3,\ldots]$ might not be the HCF of $\xi_{\bfu}$. Nevertheless, in this section we will have $|c_j|\geq \sqrt{8}$ for every $j\in \N$ which, by virtue of \ref{Prop:FullCyls:02} in Proposition \ref{Prop:FullCyls}, implies $\xi_\bfu=[c_1,c_2,c_3,\ldots]_\C$.
\end{remark}
By construction, the number $\xi_{\bfu}=[c_1,c_2,c_3,\ldots]$ has the following useful properties.  \begin{proposition}\label{Prop:Gero:Proof}
Let $m\in\N$ be arbitrary. Denote by $d_m$ the nearest Gaussian integer to $b^{v_k}\xi_{\bfu}$. The following assertions hold: 
\begin{enumerate}[label=\normalfont(\arabic*)]
\item\label{Prop:Gero:Proof:01} We have 
    \[
    c_{2^m(h+1)} = c^{u_{m-1}} -1 
    \quad\text{ and }\quad
    \frac{p_{2^m(h+1)-1}}{q_{2^m(h+1)-1}}=\frac{d_m}{b^{v_m}}.
    \]
    \item\label{Prop:Gero:Proof:02}  If $j\in \{2^{m-1}(h+1), \ldots, 2^{m}(h+1)-1\}$, then
    \[
    \left| \xi_{\bfu} - \frac{p_{j-1}}{q_{j-1}}\right|
    >
    \left(\frac{2}{2 + \sqrt{2}}\right)\,\frac{1}{ |q_j|^2 \left( |b|^{u_{ m - 1 }} + \frac{\sqrt{2}}{2}\right)}.
    \]
    \item\label{Prop:Gero:Proof:03} We have
    \[
    \frac{1}{2}\,
    \frac{1}{|b|^{v_{m+1}}}
    \leq 
    \left| \xi_{\bfu} - \frac{d_m}{b^{v_m}}\right|
    \leq 
    \frac{3}{2}\,
    \frac{1}{|b|^{v_{m+1}}}.
    \]
\end{enumerate}
\end{proposition}
\begin{proof}
    \begin{enumerate}
        \item This follows from the definition of $\bfc$.
        \item  By construction of $\bfc$, for each $j\in\{2^{m-1}(h+1),\ldots, 2^m(h+1)-1\}$ we have $|c_j| \leq |b|^{u_{m - 1}}$ and, by Lemma \ref{Le:Proof:Gero:01},
\[
\left| \xi_{\bfu} - \frac{p_{j-1}}{q_{j-1}}\right| 
>
\left(\frac{2}{2 + \sqrt{2}}\right)\,
\frac{1}{ |q_j|^2 \left( |b|^{u_{ m - 1 }} + \frac{\sqrt{2}}{2}\right)}.
\]
        \item From $4< |b|^{u_{m+2} + v_{m+1}}$ we obtain
\[
\left| \sum_{j=2}^{\infty} \frac{1}{b^{v_{m+j}}}\right|
\leq 
\sum_{j=2}^{\infty} \frac{1}{|b|^{v_{m+j}}}
<
\frac{1}{|b|^{v_{m+2}} -1 }
<
\frac{2}{|b|^{v_{m+2}}}
< \frac{1}{2|b|^{v_{m+1}} }.
\]
As a consequence, we have
\[
\frac{1}{2|b|^{v_{m+1}}}
< 
\left| \sum_{j=1}^{\infty} \frac{1}{b^{v_{m+j}}}\right|
<
\frac{3}{2|b|^{v_{m+1}}},
\quad
\text{ so }
\quad
\frac{|b|^{v_m}}{2|b|^{v_{m+1}}}
< 
\left| b^{v_m}\xi_{\bfu}  - d_m\right|
<
\frac{3|b|^{v_m}}{2|b|^{v_{m+1}}}.
\]
    \end{enumerate}
\end{proof} 
The proofs of Theorem \ref{Te:Gero:01} and \ref{Te:Gero:02} follow the lines of Bugeaud's argument for RCFs \cite{BugeaudDACS2008}. We first observe that for any $\xi\in \C\setminus \Q(i)$ and all $c\in \Z[i]\setminus\{0\}$, $r\in \Q(i)$ we have  $\mu\left( c\xi + r\right) = \mu(\xi)$.

\begin{proof}[Proof of Theorem \ref{Te:Gero:01}]
Let $v\in \N$ such that $|b|^v\geq \sqrt{8}$, then for every $z\in \Z[i]$ with $|z|\geq \sqrt{8}$ we have $(b^v,z,-b^v), (-b^v,z,b^v) \in \sfF(3)$ and, by Folding Lemma,
\[
\frac{1}{b^v} - \frac{1}{zb^{2v}}
=
[b^v,z,-b^v]_{\C}.
\]
Define $u_1:=1$ and for $n\in \N$ let $u_{n+1}$ be the minimum integer $u$ satisfying $2< |b|^{u_{n+1}} |b|^{2v_n}\Psi(|b|^{v_n})$, and hence
\begin{equation}\label{Eq:Proof:Gero:02}
\forall n\in \N
\quad
2
<
|b|^{v_{n+1}} \Psi\left( |b|^{v_n}\right)
\leq
2|b|
\;\text{ and }\;
\sqrt{8}\leq b^{v_n}.
\end{equation}
From statement \ref{Prop:Gero:Proof:03} 
of Proposition \ref{Prop:Gero:Proof}, we conclude  that
\[
\frac{\Psi\left( |b|^{v_k}\right) }{4|b|}
\leq 
\left| \xi_{\bfu}  - \frac{d_k}{b^{v_k}} \right|
<
\Psi\left( |b|^{v_k}\right).
\]
The sequence $\bfu$ is strictly increasing, because $(v_n)_{n\geq 1}$ is strictly increasing and $x\mapsto x^2\Psi(x)$ is strictly decreasing. If $j\in\N$ and $m\in\N$ is such that $2^{m-1}(h+1)\leq j$, then $|b|^{v_{m-1}}\leq |q_j|$. As a consequence, in the view of Proposition \ref{Prop:Gero:Proof}, for any $\delta>0$ and any large $j\in\N$ we have
\begin{align*}
\left( |b|^{u_{ m - 1 }} + \frac{\sqrt{2}}{2}\right)|q_j|^2\Psi(|q_j|)
&\leq 
(|b|^{u_{m-1}} + 1)|q_j|^2\Psi(|q_j|) \\
&\leq
(|b|^{u_{m-1}} + 1) |b|^{2v_{m-1}}\Psi\left(|b|^{v_{m-1}}\right) \\
&\leq 
2(|b|^{u_{m-1}} + 1) |b|^{1-u_m} \\
&\leq
\frac{2}{\sqrt{2}} (|b|+\delta).
\end{align*}
In the last step we have used $|b|> \sqrt{2}$ and $|b|^{-u_m}\to 0$ when $m\to\infty$. In particular, when $\delta>0$ satisfies $(|b| + 2\delta)^{-1}> c$, the third part of Proposition \ref{Prop:Gero:Proof} tells us that
\[
\left| \xi_{\bfu} - \frac{d_m}{b^{u_m}}\right|
> 
\frac{\Psi(|q_j|)}{|b| + 2\delta}
>
c \Psi(|q_j|).
\]
We conclude that $\xi_{\bfu}\in \KK(\Psi)\setminus\KK(c\Psi)$. In order to get uncountably many numbers when $A\geq 2$, consider sequences $\bfw=(w_n)_{n\geq 1}$ determined satisfying $w_1=1$, $w_{2n+1}=u_n$, and $w_{2n}\in \{1,2\}$ for $n\in\N$ and apply the above construction.
\end{proof}
\begin{proof}[Proof of Theorem \ref{Te:Gero:02}]
Take $\lambda>0$ and $\tau\geq 2$. The estimate $\mu( \xi_{\lambda,\tau})\geq \tau$ can be obtained approximating $\xi_{\lambda,\tau}$ by the partial sums of the series defining it. In order to show $\mu( \xi_{\lambda,\tau})\leq \tau$, define $v_0:=\lfloor \lambda\tau^{3 + n_0(\tau,\lambda)}\rfloor$  and let $\bfu=(u_n)_{n\geq 1}$ be given by
\[
\forall n\in \N
\quad
u_n
:=
\left\lfloor \lambda\tau^{n + 3 + n_0(\tau,\lambda)} \right\rfloor 
-
2\left\lfloor \lambda\tau^{n + 2 + n_0(\tau,\lambda)} \right\rfloor.
\]
It can be readily checked that
\[
\forall n\in\N
\quad
v_n = \left\lfloor \lambda\tau^{n + 3 + n_0(\tau,\lambda)} \right\rfloor.
\]
Take any $j\in\N$ and pick $m\in\N$ such that $2^m(h+1)-1 \leq j < 2^{m+1}(h+1)-1$, then
\[
|q_{j-1}|
\leq
|b|^{v_{m+1}}
= 
|b|^{\lfloor \lambda\tau^{4+m+n_0(\lambda,\tau)} \rfloor }
\leq
|b|^{\lambda\tau^{4+m+n_0(\lambda,\tau)}},
\]
so
\[
|q_{j-1}|^{\tau - 2}
\leq 
b^{(\tau -2)(\lambda\tau^{4+m+n_0(\lambda,\tau)})}.
\]
By definition of $u_{m-1}$, we have
\[
|b|^{u_{m-1}} 
\leq 
|b|\frac{|b|^{\lambda\tau^{2+m+n_0(\tau,\lambda)} }}{ |b|^{2\lambda \tau^{m+1+n_0(\tau,\lambda))}}} 
=
|b| \, |b|^{(\tau-2)\lambda^{2+m+n_0(\tau,\lambda)} }
\leq
|b| \, |b|^{(\tau-2)\lambda^{4+m+n_0(\tau,\lambda)} }.
\]
in the view of Lemma \ref{Le:Proof:Gero:01}, for some constants depending on $b$ (implied by $\gg_b$),
\[
\left| \xi_{\lambda,\tau} - \frac{p_{j-1}}{q_{j-1}}\right| 
\gg_{b}
\frac{1}{|q_{j-1}|^2|b|^{u_{m-1}}}
\gg_{b}
\frac{1}{|q_{j-1}|^{\tau}},
\]
therefore $\mu(\xi_{\lambda,\tau})\leq \tau$. 

A similar argument holds when $\tau=2$, but we need to be careful. In this case, we have 
\[
\forall k\in\N 
\quad
\varepsilon_k
:=
\lfloor \lambda 2^{k+1}\rfloor - 2\lfloor \lambda 2^k \rfloor \in \{0,1\}.
\]
Folding Lemma gives a continued fraction of the form $[a_1,\ldots, a_n, 1, -a_n, -a_{n-1}, \ldots, a_1]$, which is not a HCF. However, using the well-known identity
\begin{equation}\label{foldingforx=1}
x_1 + \cfrac{1}{1 + \cfrac{1}{x_2 + x_3}}
=
x_1 + 1 + \cfrac{1}{-(x_2+1)-x_3}
\end{equation}
(see, for example, \cite[p. 258]{IosifescuKraaikamp2002}), we obtain the HCF
\begin{equation}\label{pm1}
[a_1,\ldots, a_{n-1}, a_n + 1, a_{n} -1 , a_{n-1}, \ldots, a_{1}]_{\C}
\end{equation}
(see Proposition \ref{Prop:FullCyls}). Moreover, $\mu(\xi_{2,\tau})=2$ because the partial quotients are bounded. 
\end{proof}
\section{Zaremba conjecture for powers of small numbers}\label{SEC:Zaremba}
First, let us introduce some useful notation. For $z\in\Q(i)\cap\mfF$ with a HCF $z=[a_1(z),\ldots,a_n(z)]_{\C}$, 
we define
$$K_\C(z) = \max (|a_1(z)|,\ldots,|a_n(z)|)$$ and
$a_{last}=a_{last}(z)$ to be the last partial quotient in HCF of $z$ (in this case $a_{last}(z)=a_n(z)$). 


In the previous section in proofs of Theorem \ref{Te:Gero:01} and Theorem \ref{Te:Gero:02} we have considered denominators of the form $b^k$, where $b=-A\pm i$ for some $A,k\in\N$. Using Folding Lemma we can show that if we fix $A$ to be a small integer, we can always find HCF with bounded partial quotients. For $A=3$ we have the following theorem.
\begin{theorem}\label{zarembafor-3+i}
    If $m=(-3\pm i)^k,k\geq1$, then there exists a Gaussian integer $a$ with $\gcd(a,m)=1$, $a/m\in\mfF$, such that $K_\C(a/m) \leq 3\sqrt2$.
\end{theorem}
\begin{proof}
For the base of inductive procedure consider the following HCFs.
\begin{align*}
\frac{1}{-3\pm i}=[-3\pm i]_\C, \quad \frac{2\pm3i}{(-3\pm i)^2}=[\mp 3i,-2\mp3i]_\C.
\end{align*}
Then for even powers of $-3\pm i$ we apply Lemma \ref{folding} with $x=1$, so the resulting fraction will be of the form 
$$
[a_1,\ldots,a_{last-1},a_{last}+1,a_{last}-1,a_{last-1},\ldots,a_1]_\C.
$$
Note that we will only change the real part of the last partial quotient by $\pm1$ on each step of the iteration and as one can see, it will not impose any non-trivial restrictions on the partial quotients as for every $i$ we have $|a_i| \geq \sqrt8.$

For odd powers of $-3\pm i$, we apply Lemma \ref{folding} with $x=(-3\pm i)$, so the resulting fraction will be of the form 
$$
[a_1,\ldots,a_{last},-3\pm i,-a_{last},\ldots,-a_1]_\C
$$
and once again for every partial quotient we have $|a_i|\geq \sqrt8$, so there are no non-trivial restrictions on the partial quotients.

Note that for both odd and even powers, the maximum absolute value of partial quotients will be not greater than $3\sqrt2.$
    Continuing the procedure, we will generate HCFs for every power of $-3\pm i$.
\end{proof}
We also have a similar statement for $b=-2\pm i$.
\begin{theorem}\label{zarembafor-2+i}
    If $m=(-2\pm i)^k,k\geq1$, then there exists a Gaussian integer $a$ with $\gcd(a,m)=1$, $a/m\in\mfF$, such that $K_\C(a/m) \leq 3\sqrt2$.
\end{theorem}
\begin{proof}
For the base of inductive procedure consider the following HCFs.
\begin{align*}
&\frac{1}{-2\pm i}=[-2\pm i]_\C, \quad \frac{\pm2i}{(-2\pm i)^2}=[-2\mp i,\pm 2i]_\C, \quad \frac{2\mp 4i}{(-2\pm i)^3}=[-2\pm i, -2\pm i, 2\mp i]_\C, \\ 
& \frac{5\mp6i}{(-2\pm i)^4} = [2\mp 3i, -1\mp 2i,-3\pm i]_\C, \quad \frac{13 \mp 11i}{(-2\pm i)^5} = [\pm 3i,1\mp 3i,2\mp i,-2\mp 2i]_\C, \\
& \frac{27 \mp 38 i}{(-2 \pm  i)^6} = [-1\mp 3i, 1\mp 2i,1\mp 2i, -3\mp 2i,2 \mp 3i]_\C,\\
& \frac{\mp 97i}{(-2\pm i)^7} = [\pm 3i,3 \pm 3i, -2 \mp 2i, \pm 3i, 1\mp 3i]_\C.
\end{align*}
Note that in all fractions $p_n/(-2\pm i)^n,\, n \geq4$, all partial quotients satisfy 
\begin{equation}\label{allpartial}
\sqrt5 \leq \left|a_i\left(\frac{p_n}{(-2\pm i)^n}\right) \right|\leq 3\sqrt2
\end{equation}
and that the first and the last partial quotients satisfy
\begin{equation}\label{firstlast}
\sqrt5 \leq \left|a_i\left(\frac{p_n}{(-2\pm i)^n}\right) \pm 1 \right|\leq 3\sqrt2 \quad \text{for } i\in\{1,\rm last\}.
\end{equation}
Note that if $a_i>\sqrt5$, then there are no non-trivial restrictions on $a_{i+1}$. When $a_i=\sqrt5$, as per comments after Table \ref{Table:HensleyTable} we need to be careful in regards to whether the resulting fraction is a valid HCF. We will check the validity in Section \ref{validitycheck}.


For even powers of $(-2\pm i)^{2k},\,k\geq4$ we apply Lemma \ref{folding} with $x=1$ and $q_n = (-2\pm i)^{k}$, where by induction we already have a suitable HCF for $q_n =(-2\pm i)^{k}$. The resulting fraction will be of the form 
$$
\frac{p'}{(-2\pm i)^{2k}}=[a_1,\ldots,a_{last-1},a_{last}+1,a_{last}-1,a_{last-1},\ldots,a_1]_\C.
$$
It is easy to see that properties \eqref{allpartial} and \eqref{firstlast} are preserved under the application of Lemma \ref{folding} for this fraction, hence we can continue the inductive procedure.

For odd powers of $(-2\pm i)^{2k+1},\,k\geq4$, we apply Lemma \ref{folding} with $x=(-2\pm i)$, so the resulting fraction will be of the form 
$$
\frac{p'}{(-2\pm i)^{2k+1}}=[a_1,\ldots,a_{last},-2\pm i,-a_{last},\ldots,-a_1]_\C
$$

Once again, properties \eqref{allpartial} and \eqref{firstlast} are preserved.
    Continuing the procedure, we will generate continued fractions for every power of $-2\pm i$ with partial quotients bounded by $3\sqrt2.$ The final step is to show that the resulting fractions provided by Folding Lemma are indeed valid HCFs. As the proof of this fact is very technical, we provide details in a separate Section \ref{validitycheck}.
\end{proof}

Due to the statements \eqref{remarkdiff3} and \eqref{remarkdiff4} of Remark \ref{remarkdiff}, we can deal with HCFs of real rational numbers, too.

\begin{theorem}\label{zarembafor2}
If $m=2^k,k\geq1$, then there exists an odd integer $1\leq a<m$ such that $K_\C(a/m) \leq~8$.
\end{theorem}
\begin{proof}
First, note that 
\begin{align*}
&-\frac{1}{2^1}=[-2]_\C, \quad \frac{1}{2^2}=[4]_\C, \quad \frac{3}{2^3}=[3,-3]_\C, \quad \frac{5}{2^4}=[3,5]_\C, \quad \frac{9}{2^5} = [4,-2,-4]_\C, \quad \frac{17}{2^6}=[4,-4,-4]_\C,\\
&\frac{19}{2^7}=[7,-4,5]_\C, \quad \frac{79}{2^8}=[3,4,6,3]_\C, \quad \frac{71}{2^9}=[7,5,-4,-4]_\C, \quad \frac{165}{2^{10}}=[6,5,-7,5]_\C, \\
& \frac{423}{2^{11}}=[5,-6,-3,-5,-4]_\C, \quad \frac{557}{2^{12}}=[7,3,-6,5,-7]_\C, \quad \frac{1453}{2^{13}}=[6,-3,4,5,-5,-5].
\end{align*}
Note that all rational numbers $r/2^k,k\geq6$ listed above satisfy the conditions of the statement \eqref{remarkdiff4} of Remark \ref{remarkdiff}, and so these HCFs are reversible; we also have $3\leq|a_1(r/2^k)|,|a_{last}(r/2^k)|\leq 7$.

To prove the statement of this theorem, it is remaining to introduce an inductive procedure, which will generate valid HCF expansions for the fractions $g/2^N$ for some $g < 2^{N-1}$ and all $N\geq14.$

We are using the following inductive procedure: for every odd integer $N=2k+3\geq15$ we will apply Lemma \ref{folding} with $x=8$ and $q_n=2^k$. We get that $k\geq6$ and by the induction hypothesis there exists a reduced fraction $r/2^k$ with $K_\C(r/2^k) \leq 8$ and $3\leq|a_1(r/2^k)|,|a_{last}(r/2^k)|\leq 7$, so we also let $p_n=r$. Then the right-hand side in \eqref{folding:eq} is a fraction with denominator $xq_n^2 = 2^3\cdot 2^{2k} = 2^N,$ and its continued fraction expansion is
\begin{equation}\label{explicit3}
\frac{r'}{2^N}=[a_1,\ldots,a_{last},8,-a_{last},\ldots,-a_1]_\C.
\end{equation}
Note that this is indeed a valid HCF, because, as we said previously, $[a_1,\ldots,a_{last}]_\C$ is reversible, $3\leq|a_1(r/2^k)|,|a_{last}(r/2^k)|\leq 7$ and we added new partial quotients equal to $8$ in the middle, which does not impose any restrictions as per Remark \ref{remarkdiff}. Notice that the continued fraction \eqref{explicit3} preserves the property $3\leq|a_1(r'/2^N)|,|a_{last}(r'/2^N)|\leq 7$ and $K_\C(r'/2^N)\leq 8$, so we indeed can continue our inductive procedure.

For every even integer $N=2k+2\geq14$, we will apply Lemma \ref{folding} with $x=4$ and $q_n=2^k$. Note that in this case we also have $k\geq6$ and by induction, there exists a reduced fraction $r/2^k$ with $K_\C(r/2^k) \leq 8$ and $3\leq|a_1(r/2^k)|,|a_{last}(r/2^k)|\leq 7$, so we also let $p_n=r$. This means that the application of Folding Lemma will generate a valid HCF expansion. The right-hand side in \eqref{folding:eq} is a fraction with denominator $xq_n^2 = 2^2\cdot2^{2k}=2^N,$ and its continued fraction expansion is
$$
\frac{r'}{2^N}=[a_1,\ldots,a_{last},4,-a_{last},\ldots,-a_1]_\C,
$$
so this continued fraction preserves the property $3\leq|a_1(r'/2^N)|,|a_{last}(r'/2^N)|\leq 7$ and also $K_\C(r'/2^N)\leq 8$, so once again we can continue the procedure.
\end{proof}
\begin{remark}
This is similar to the idea of Niederreiter \cite{Niederreiter1986} for the case of RCFs of real numbers, however due to statement \eqref{remarkdiff1} of Remark \ref{remarkdiff}, in that case, one does not need to worry about reversibility as there are no non-trivial restrictions on partial quotients in RCFs. We also note that Zaremba's conjecture for powers of $2$ in the case of RCFs can be reformulated and proved in terms of Minkowski question mark function $?(x)$, see \cite{MR4121875}.
\end{remark}
\begin{remark}\label{remark2n}
    Note that we cannot let $x=2$ or $x=-2$ in the proof of Theorem \ref{zarembafor2}, because by Folding Lemma we would have $[\ldots,a_{last},\pm2,-a_{last},\ldots]$ in the middle, which is not reversible. This is due to the fact that there are numbers of different signs around $\pm2$ and hence by statement \eqref{remarkdiff3} of Remark \ref{remarkdiff} (or by Table \ref{Table:HensleyTable}) it is not reversible.
\end{remark}
Similarly, we provide a result for powers of $3$.
\begin{theorem}\label{zarembafor3}
If $m=3^k,k\geq1$, then there exists an integer $1\leq a<m$ with $\gcd(a,3)=1$, such that $K_\C(a/m) \leq 8$.
\end{theorem}
\begin{proof}
First, we see that
\begin{align*}
&\frac{1}{3^1}=[3]_\C, \quad \frac{4}{3^2}=[2,4]_\C, \quad \frac{10}{3^3}=[3,-3,-3]_\C, \quad
 \frac{19}{3^4}=[4,4,-5]_\C, \\ 
 & \frac{50}{3^5} = [5,-7,-7]_\C, \quad \frac{107}{3^6}=[7,-5,-3,7]_\C, \quad
\frac{323}{3^7}=[7,-4,-3,4,-7]_\C.
\end{align*}
Note that all fractions with denominators $3^k$ for $k=4,5,6,7$ listed above satisfy the condition of statement \eqref{remarkdiff4} of Remark \ref{remarkdiff} and so the reversed continued fraction is also a valid HCF.

Once again, we will use two inductive procedures: one for odd powers and one for even powers.

For $N=2k\geq8$ we will apply Lemma \ref{folding} with $x=1$ and $q_n=3^k$. Note that in this case $k\geq4$ and by the induction hypothesis there exists a reduced fraction $r/3^k$ with $K_\C(r/3^k)\leq 8$ and $4\leq |a_1(r/3^k)|,|a_{last}(r/3^k)|\leq 7$.

So applying Lemma \ref{folding}, we get that the right-hand side in \eqref{folding:eq} is a fraction with denominator $xq_n^2=1\cdot3^{2k}=3^N$, and by the identity \eqref{foldingforx=1} its Hurwitz continued fractions expansion is
\begin{equation}\label{explicit1}
\frac{r'}{3^N}=[a_1,\ldots,a_{last-1},a_{last}+1,a_{last}-1,a_{last-1},\ldots,a_1]_\C.
\end{equation}
We note that $a_1(r'/3^N)=a_{last}(r'/3^N) = a_1(r/3^k)$ and so
$4\leq |a_1(r'/3^N)|,|a_{last}(r'/3^N)|\leq 7$, so this property of $a_1(z)$ and $a_{last}(z)$ preserves under this procedure. Because of the statement \eqref{remarkdiff4} in Remark \ref{remarkdiff}, we see that this continued fraction is also reversible as a HCF. Also, note that from the explicit form of the continued fraction \eqref{explicit1}, it is clear that $K_\C (r'/3^N) = \max (K_\C (r/3^k), |a_{last}(r/3^k)|+1)\leq 8$.

For $N=2k+1\geq9$ we will apply Lemma \ref{folding} with $x=3$ and $q_n=3^k$. Note that in this case $k\geq4$ and by the induction hypothesis there exists a reduced fraction $r/3^k$ with $K_\C(r/3^k)\leq 8$ and $4\leq |a_1(r/3^k)|,|a_{last}(r/3^k)|\leq 7$.

So applying Lemma \ref{folding}, we get that the right-hand side in \eqref{folding:eq} is a fraction with denominator $xq_n^2=3\cdot3^{2k}=3^N$, and its Hurwitz continued fractions expansion is
\begin{equation}\label{explicit2}
\frac{r'}{3^N}=[a_1,\ldots,a_{last},3,-a_{last},\ldots,-a_1]_\C.
\end{equation}
We note that $a_1(r'/3^N)=-a_{last}(r'/3^N) = a_1(r/3^k)$ and so
$4\leq |a_1(r'/3^N)|,|a_{last}(r'/3^N)|\leq 7$, so this property of $a_1(z)$ and $a_{last}(z)$ preserves under this procedure. Because of the statement \eqref{remarkdiff4} in Remark \ref{remarkdiff}, we see that this continued fraction is also reversible as a HCF. Also, note that from the explicit form of the continued fraction \eqref{explicit2}, it is clear that $K_\C (r'/3^N) =  K_\C (r/3^k)\leq 8$.

Hence, continuing this procedure, we can get valid HCF expansions for denominators of the form $3^k,k\geq1$ with 
$K_\C (r/3^k)\leq 8$.
\end{proof}
For powers of $5$ we have a slightly better statement.
\begin{theorem}\label{zarembafor5}
    If $m=5^k,k\geq1$, then there exists an integer $1\leq a<m$ with $\gcd(a,5)=1$, such that $K_\C(a/m) \leq 7$.
\end{theorem}
\begin{proof}
    We start with seed fractions
$$
\frac{1}{5}=[5]_\C, \quad \frac{6}{25}=[4,6]_\C
$$
and then proceed as in Theorem \ref{zarembafor3} with $x=5$ for odd powers and $x=1$ for even powers.
\end{proof}


\section{Verification of the validity of the sequences for powers of $(-2+i)$}\label{validitycheck}
We focus on $(-2+i)^4$, the other powers are handled similarly. In Section \ref{SEC:HCF}, we said that for a given valid sequence $\bfa=(a_1,\ldots, a_{n})$, $n\in\N$, and any $b\in\DD$ the validity of $\bfa b$ may depend on the entire sequence $\bfa$. Nevertheless, if $\mfFc_j(a_1,\ldots,a_{j+1})= \mfFc$ for some $j\in\{1,\ldots, n-1\}$, the values of $b$ making $\bfa b$ valid are independent of $(a_1,\ldots, a_{j})$. This observation---a trivial corollary of part \ref{Prop:ValidSeq:01} of Proposition \ref{Prop:ValidSeq} below---lies in the heart of our proof. The first point in Proposition \ref{Prop:ValidSeq} follows directly from the definition of cylinders and prototype sets (cfr. \cite[Lemma 2.3]{HeXiong2022}). Furthermore, since the functions $\tau_b$, for $b\in \DD$, and $\iota$ are homeomorphisms, we can replace $\mfF$ and $\CC$ with $\mfFc$ and $\CC^{\circ}$, respectively. The second point is contained in the proof of \cite[Proposition 4.5]{BugeaudGeroHussain2023}.
\begin{proposition}\label{Prop:ValidSeq}
Take $m\in\N$ and ${\bfx} =(x_1,\ldots, x_m)\in \Omega(m)$.
\begin{enumerate}[label=\normalfont(\arabic*)]
\item\label{Prop:ValidSeq:01} We have $\mfF_m( {\bfx})=\tau_{-x_m}\iota(\mfF_{m-1}(x_1,\ldots, x_{m-1})\cap \CC_1(x_m))$.
\item\label{Prop:ValidSeq:02} If $\mfF^{\circ}_m( {\bfx})=\mfF^{\circ}$, then $\mfF^{\circ}_m(\overleftarrow{-\bfx})=\mfF^{\circ}_m(\overline{x_1},\ldots, \overline{x_m})=\mfF^{\circ}$.
\end{enumerate}
\end{proposition}
\begin{remark}
    Given $a,b\in \DD$, we may compute $\mfFc_2(a,c)$ using part \ref{Prop:ValidSeq:01} of Proposition \ref{Prop:ValidSeq} and $\iota(\CCc_1(b))=\iota(\mfF)\cap \tau_b(\mfF)$. 
\end{remark}
For each $ {\mathbf{x}}=(x_1,\ldots, x_{last})\in \DD^{+}$, let $|\mathbf{x}|$ be the length of $ {\mathbf{x}}$, and denote
\begin{align*}
& {\mathbf{x}}^+:=(x_1,\ldots, x_{last - 1}, x_{last} + 1), \\
& {\mathbf{x}}_{-} :=(x_1 -1, x_2\ldots, x_{last}), \\
&\overleftarrow{ \mathbf{x}}_{-} := \left(\overleftarrow{\mathbf{x}} \right)_{-}.
\end{align*}
Consider $d=-2+i$, $ f(\mathbf{x}):={f_{d}(\mathbf{x})}$ and $g(\mathbf{x}) := {\mathbf{x}}^{+}\overleftarrow{ \mathbf{x}}_{-}$. For $h_1,\ldots, h_n\in\{f,g\}$, we write $h_n\cdots h_1:=h_n\circ \cdots \circ h_1$. Let $ {\bfa}=(a_1,a_2,a_3)$ be the partial quotients of $(5-6i)/(-2+i)^4$; that is
\[
 {\bfa} := (2 - 3i, -1 - 2i,-3+ i).
\]
From Table \ref{Table:HensleyTable} we get $\mfFc_3(\overleftarrow{\bfa})=\mfFc$ and we can easily check that $\mfFc_3(\overleftarrow{\bfa} )=\mfFc$. Consider $ {\bfc}=(c_1,\ldots, c_7)$ given by
\[
 {\bfc}
:=
f( {\bfa})
=
(2 - 3i, -1 - 2i, -3 + i,
-2 + i,
3- i, 1 + 2i, -2 + 3i),
\]
From $\mfFc_3(2 - 3i, -1 - 2i, -3 + i) = \mfFc$, $\mfFc_2(-2+i, 3-i)= \mfFc$, and $\mfFc_2(1+2i, -2 + 3i)=\mfFc$ we obtain $\mfFc_{|f( {\bfa})|}({f( {\bfa})})=\mfFc$. Indeed, using part \ref{Prop:ValidSeq:01} of Proposition \ref{Prop:ValidSeq}, we arrive at
\begin{align*}
\mfFc_4(c_1,c_2,c_3,c_4)
&=
\tau_{-c_4}\iota \left( \mfFc_3(c_1,c_2,c_3)\cap \CCc_1(c_4)\right) \\
&=
\tau_{-c_4}\iota \left( \mfFc_3(2 - 3i, -1 - 2i, -3 + i)\cap \CCc_1(-2+i)\right)
=
\mfFc_1(-2+i),
\end{align*}
hence
\begin{align*}
\mfFc_5(c_1,c_2,c_3,c_4,c_5)
&=
\tau_{-c_5}\iota \left( \mfFc_4(c_1,c_2,c_3, c_4)\cap \CCc_1(c_5)\right) \\
&=
\tau_{-c_5}\iota \left( \mfFc_1(-2+i) \cap \CCc_1(3-i)\right)
=
\mfFc_2(-2+i, 3-i)
=
\mfFc
\end{align*}
and
\begin{align*}
\mfFc_7(\bfc)
&=
\tau_{-c_7}\iota \left( \mfFc_6(c_1,\ldots, c_6)\cap \CCc_1(c_7)\right) \\
&=
\tau_{-c_7}\iota \left( \tau_{-c_6}\iota\left( \mfFc_5(c_1,\ldots, c_5)\cap \CCc_1(c_6)\right)\cap \CCc_1(c_7)\right) \\
&=
\tau_{-c_7}\iota \left( \mfFc_1(c_6) \cap \CCc_1(c_7)\right)  = 
\mfF_2(c_6,c_7) = \mfFc_2(1+2i, -2 + 3i)=\mfFc.
\end{align*}
Similarly, $\mfFc_7\left(\overleftarrow{\bfc}\right)=\mfFc$ is implied by 
\[
\mfFc_3(-3+i, -1-2i, 2-3i)= \mfFc, \quad
\mfFc_2(-2 + i, -3 + i) = \mfFc, \quad
\mfFc_2(-1 - 2i, 2 - 3i) = \mfFc. 
\]
We may show $\mfFc_6( {g(\bfa)})=\mfFc_6(2-3i, -1-2i, -2+i, -4+i, -1-2i,  2-3i)=\mfFc$ using
\[
\mfFc_2(2-3i, -1-2i)=\mfFc(-1-2i),\quad
\mfFc_2(-1-2i, -2+i)=\mfFc(-2 + i),
\]
\[
\mfFc_2(-2+i, -4+i)=\mfFc, \quad
\mfFc_2(-1-2i,  2-3i)=\mfFc.
\]
Finally, $\mfFc_6(\overleftarrow{g(\bfa)})=\mfFc$ is a consequence of
\[
\mfFc_2(2-3i, -1-2i)=\mfFc_1(-1-2i), \quad
\mfFc_2(-1-2i,-4+i)=\mfFc, \quad
\]
\[
\mfFc_2(-2+i,-1-2i) = \mfFc_1(-1-2i), \quad
\mfFc_2(-1-2i,2-3i)=\mfFc.
\]
Let $N\in\N_{\geq 2}$. For $j\in\{1,\ldots, N\}$ and $\bfh=(h_1,\ldots, h_j)\in\{f,g\}^j$, put $\bfh( {\bfa}):=h_j\cdots h_1(\bfa)$ and $n(\bfh)=|\bfh( {\bfa})|$. Assume that
\begin{equation}\label{Eq:IndHyp}
\forall j\in\{1,\ldots, N\}
\quad
\forall \bfh\in \{f,g\}^j
\quad
\mfFc_{n(\bfh)}\left(\bfh( {\bfa})\right)
=
\mfFc
\quad\text{ and }\quad
\mfFc_{n(\bfh)}\left( \overleftarrow{-\bfh(\bfa)} \right)
=\mfFc.
\end{equation}
Take any $\bfh=(h_1,\ldots, h_N)\in \{f,g\}^N$ and $h\in \{f,g\}$. Let us show that \eqref{Eq:IndHyp} will also hold for $N+1$. Note that the equality for $\overleftarrow{\bfh(\bfa)}$ is equivalent to the equality for $\overleftarrow{-\bfh( {\bfa})}$ by part \ref{Prop:ValidSeq:02} of Proposition \ref{Prop:ValidSeq}.

For the rest of the paper, we write $ {\bfb}=h_{N-1}\cdots h_1( {\bfa})$. We consider four cases.
\begin{enumerate}[label=\normalfont(\arabic*)]
\item Case $(h_N,h) = (f,f)$. Consider $ {\bfc} := {ff(\bfb)} = \bfb \, d \, \overleftarrow{-\bfb} \, d \, \bfb \, (-d) \, \overleftarrow{-\bfb}$. By \eqref{Eq:IndHyp}, we have 
\[
\mfFc_{n(\bfh )}\left(f(\bfb)\right)
= 
\mfFc
\quad\text{ and }\quad
\mfFc_{n(\bfh) } \left( \overleftarrow{-f(\bfb)} \right)
=
\mfFc.
\]
The first term of $ {\bfb}$ is $a_1$ and, since $\mfFc_2(d,a_1)=\mfFc_2(-2+i, 2-3i)=\mfFc$, we conclude from statement \ref{Prop:ValidSeq:01} of Proposition \ref{Prop:ValidSeq}. that $\mfF_{n(\bfh f) }({\bfh f (\bfa)})=\mfFc$. In order to prove $\mfFc_{n(\bfh f)}(\overleftarrow{-\bfc})=\mfFc$, we note that $\overleftarrow{-\bfc} =  {\bfb}\, d\, \overleftarrow{-\bfb}\, (-d)\,  {\bfb}\, (-d)\, \overleftarrow{-\bfb}$, that $\mfFc_2(-d,a_1)=\mfFc_2(2-i,2-3i)=\mfFc$, and then argue as we did for $\bfc$.

\item Case $(h_N,h) = (g,f)$. Consider $ {\bfc}:= fg(\bfb) = g(\bfb)\,d\, \overleftarrow{-g(\bfb)}$. In order get $\mfFc_{n(\bfh f)}(\bfh f(\bfa))=\mfFc$, first observe that
\begin{equation}\label{Eq:Case:gf}
\mfFc_{n(\bfh)}(g(\bfb))=\mfFc, 
\quad
\mfFc_{n(\bfh)}\left( \overleftarrow{-g(\bfb)}\right)=\mfFc,
\end{equation}
 by \eqref{Eq:IndHyp}. Since the last letter of $\overleftarrow{\bfb}_{-}$ is $a_1$, the first letter of $\overleftarrow{( {\bfb}^+ \overleftarrow{\bfb}_{-} )}_{-} $ is $-a_1$ and the desired equality follows from \eqref{Eq:Case:gf} and $\mfFc_2(d,-a_1)=\mfFc_2(-2+i,-2+3i)=\mfFc$. The same observations and $\mfFc_2(-d,-a_1)=\mfFc_2(2-i,-2+3i)=\mfFc$ imply $\mfFc_{n(\bfh f)} (\overleftarrow{-\bfc } ) = \mfFc$.

\item Case $(h_N,h) = (f,g)$. Write $\bfc:=gf(\bfb)$, so
\[
\bfc
=
f(\bfb)^{+}\, \overleftarrow{f(\bfb)}_{-}.
\]
Therefore, in order to show $\mfFc(\bfc)=\mfFc$ it suffices to verify
\begin{equation}\label{EQ:Case:fg}
    \mfFc_{n(\bfh)}(f(\bfb)^{+})=\mfFc
    \quad\text{ and }\quad
    \mfFc_{n(\bfh)}(\overleftarrow{f(\bfb)}_{-})=\mfFc.
\end{equation}

By \eqref{Eq:IndHyp}, we have $\mfFc_{n(\bfh)}(f(\bfb))=\mfFc$; hence, by statement \ref{Prop:ValidSeq:01} of Proposition \ref{Prop:ValidSeq},
\[
\mfFc_{2n(\bfh)}(\bfb \, d \, (-b_{last},\ldots, -b_{2})) \neq \varnothing
\; \text{ and }\;
\mfFc_{2n(\bfh)-1}(\bfb \, d \, (-b_{last},\ldots, -b_{3}, )) \neq \varnothing.
\]
Since the last terms of $f(\bfb)$ are $(-b_3, -b_2, -b_1)=(-a_3, -a_2, -a_1)$, the last terms of $f(\bfb)^+$ are $(-a_3, -a_2, -a_1+1)$. Then, in the view of $-a_3= 3-i$, we have
\[
\mfFc_{2n(\bfh)-1}(\bfb \, d \, (-b_{last},\ldots, -b_{3})) = \mfFc,
\quad\text{thus, }\quad
\mfFc_{2n(\bfh)}(\bfb \, d \, (-b_{last},\ldots, -b_{2})) = \mfFc_1(-b_2).
\]
Lastly, because of $\mfFc_2(-a_2,-a_1+1)=\mfFc_2(1+2i,-1+3i)=\mfFc$, we may conclude $\mfFc_{n(\bfh)}(f(\bfb)^{+})=\mfFc$. The second equality in \eqref{Eq:Case:gf} follows from a similar observations and
\[
\mfFc_{n(\bfh)}(-b_1,\ldots, -b_{last-1}, -b_{last}+1)
=
\mfFc.
\]
Consider $\mathbf{d}= \overleftarrow{f(\bfb)}_{-}$. We have $\mfFc_{n(\bfh)}(\mathbf{d})=\mfFc$ from $-b_1-1=-a_1-1$, $\mfFc_1(-a_1) = \mfFc_1(-a_1-1)=\mfFc$ and statement \ref{Prop:ValidSeq:02} of Proposition \ref{Prop:ValidSeq}. The same observations and \eqref{Eq:IndHyp} yield $\mfFc_{n(\bfh)}(\overleftarrow{\mathbf{d}})=\mfFc$.

\item Case $(h_N,h) = (g,g)$. Put $\bfc := gg(\bfb)$. As before, it suffices to show
\[
\mfFc_{n(\bfh)}\left(g(\bfb)^{+}\right)
=
\mfFc_{n(\bfh)}\left(\overleftarrow{g(\bfb)^{+}}\right)
=
\mfFc
\quad\text{ and }\quad
\mfFc_{n(\bfh)}\left( \overleftarrow{g(\bfb)}_{-} \right)
=
\mfFc_{n(\bfh)}\left(\overleftarrow{\overleftarrow{g(\bfb)}_{-}} \right)
=
\mfFc.
\]
These identities are proven as in the previous case. 

This shows \eqref{Eq:IndHyp} for $N+1$. Therefore, if $\bfc=(c_1,\ldots, c_{last})=\bfh(\bfa)$ for some $\bfh\in\{f,g\}^{n}$, $n\in\N$, then
\[
\iota\tau_{c_{last}} \cdots \iota\tau_{c_1}(0)
=
[c_1,c_2,\ldots, c_{last} ]\in \CCc_{n(\bfh)}(\bfc) \subseteq \mfFc
\]
and $\bfc$ is valid.
\end{enumerate}

\end{document}